\documentclass{aims}
\usepackage{amsmath}
  \usepackage{paralist}
  \usepackage{graphics} 
  \usepackage{epsfig} 
\usepackage{graphicx}  \usepackage{epstopdf}
 \usepackage[colorlinks=true]{hyperref}
\hypersetup{urlcolor=blue, citecolor=red}

\usepackage{graphicx}
\usepackage{caption}
\usepackage{subcaption}
\usepackage{amsfonts}
\usepackage{amssymb}
\usepackage{fancyhdr}
\usepackage{titlesec}
\usepackage{indentfirst}
\usepackage{booktabs}
\usepackage{verbatim}
\usepackage{color}
\usepackage{tcolorbox}
\usepackage{mathtools}
\usepackage{bm}
\usepackage{amsthm}
\usepackage{dsfont}
\usepackage{wasysym}
\usepackage{empheq}
\graphicspath{{figures/}} 
\usepackage[font=small,labelfont=bf]{caption} 
\usepackage{chngcntr}
\newcommand{\rom}[1]{\uppercase\expandafter{\romannumeral #1\relax}}
\usepackage[page,header]{appendix}

\usepackage{titletoc}

\numberwithin{equation}{section}
\newtheorem{theorem}{Theorem}[section]
\newtheorem{lemma}[theorem]{Lemma}
\newtheorem{corollary}[theorem]{Corollary}

\newtheorem{definition}[theorem]{Definition}

\newtheorem{condition}[theorem]{Assumption}
\newtheorem{gallarynotation}[theorem]{Gallery of Notations}

\def\e{\epsilon}

\def\lam{\lambda}
\def\o{\omega}

\def\s{\sigma}

\def\L{\Lambda}

\def\O{\Omega}

\def\R{\mathbb R}

\def\A{\mathbb{A}}

\def\mL{\mathcal{L}}
\def\mN{\mathcal{N}}

\def\mF{\mathcal{F}}

\def\l{\left}
\def\r{\right}
\def\la{\left\langle}
\def\ra{\right\rangle}
\def\ll{\left\lVert}
\def\rl{\right\rVert}
\def\lv{\left\lvert}
\def\rv{\right\rvert}
\def\({\left(}
\def\){\right)}
\def\[{\left[}
\def\]{\right]}

\def\pt{\partial}

\def\qd{\quad}

\def\sM{\sqrt{M}}

\def\koo{k_{00}}
\def\kol{k_{01}}
\def\klo{k_{10}}
\def\kll{k_{11}}
\def\A{A}
\def\B{B}
\def\Ax{A_x}
\def\Bx{B_x}

\def\currentvolume{X}
 \def\currentissue{X}
  \def\currentyear{200X}
   \def\currentmonth{XX}
    \def\ppages{X--XX}
     \def\DOI{10.3934/xx.xx.xx.xx}

\title[Boundary Control of VFP Equations] 
      {Boundary Control of Vlasov--Fokker--Planck Equations}

\author[Michael Herty  and  Shi Jin and Yuhua Zhu ]{}

\subjclass{Primary: 58F15, 58F17; Secondary: 53C35.}
 \keywords{Lyapunov function, Feedback stabilization, Vlasov--Fokker--Planck Equations, Hypocoercivity}

 \email{herty@igpm.rwth-aachen.de}
 \email{shijin-m@sjtu.edu.cn}
 \email{yuhuazhu@stanford.edu}

\thanks{}

\begin{document}
\maketitle

\centerline{\scshape Michael Herty}
\medskip
{\footnotesize
 \centerline{Department of Mathematics at RWTH Aachen University}
   \centerline{Templergraben 55}
   \centerline{ Aachen, 52064,  GERMANY.}
} 

\medskip

\centerline{\scshape Shi Jin}
\medskip
{\footnotesize
 \centerline{Institute of Natural Sciences, and MOE-LSC}
   \centerline{Shanghai Jiao Tong University}
   \centerline{Shanghai, 200240, China}
} 

\medskip

\centerline{\scshape Yuhua Zhu}
\medskip
{\footnotesize
 \centerline{Department of Mathematics at Stanford University}
   \centerline{450 Jane Stanford Way}
   \centerline{California, 94538, USA}
}
\bigskip


\begin{abstract}
We introduce a novel Lyapunov function for stabilization of linear Vlasov--Fokker--Planck type equations with stiff source term. Contrary to existing results relying on transport properties to obtain stabilization, we present results based on hypocoercivity analysis for the Fokker--Planck operator. The existing estimates are extended to derive suitable feedback boundary control to guarantee the exponential stabilization. Further, we study  the associated macroscopic limit  and derive  conditions on the feedback boundary control such that in the formal limit  no  boundary  layer exists. 
\end{abstract}

\section{Introduction.}
\label{sec: intro}
We are interested in stabilization of kinetic partial differential equations by boundary feedback laws and its associated macroscopic equations. As a prototype we consider the Vlasov--Fokker--Planck (VFP) equation with small Knudsen number and we derive suitable conditions on the boundary control to obtain damping of perturbations of the steady state exponentially fast in time. Kinetic partial differential equations belong to the class of  hyperbolic balance laws and enjoy the particular feature of linearity in transport direction. However, the studied class  usually has a (stiff) source term described by a  small Knudsen number. As the Knudsen number tends to zero, the solution to the VFP equation converges to  a solution  to (macroscopic) partial differential equations for mass and momentum \cite{poupaud2000parabolic}. We are interested in stabilization results for {\it all} ranges of Knudsen numbers.
\par 
The   theoretical and numerical discussion of  stabilization properties of hyperbolic balance laws has recently been of interest in the mathematical and engineering community and we refer
 to \cite{O1,O2} for a survey and more references. Typically, its application has been on the stabilization of flows (on networks) governed by shallow--water or isothermal equations ~\cite{G1,G2,G3,G4,W1,W2,W3,W4,W5,W6}.  The {\em core underlying tool} for the study of these problems are Lyapunov functions for  deviation from steady states in suitable norms, e.g.~$L^2_x$, $H^2_x$. A major breakthrough has been the design of suitable weighted Lyapunov functions that allow for an exponential decay (in time) of the such deviations. Exponential decay of a continuous Lyapunov function under the so-called \textit{dissipative} boundary conditions has been proven in \cite{L1,L2,L4,L5}. Comparisons to other stability  concepts are presented in \cite{L7}.
Stability with respect to the $H^2_x$-norm yields stability of the nonlinear system~\cite{L5,O1}. Recently, similar results for numerical schemes have been established \cite{D1,D2} and the theory has been extended to balance laws 
with mixed source terms \cite{HY}.  In all cases, the exponential decay is obtained as {\it interplay } of the weights of the Lyapunov function and the linear transport property of the underlying linearized hyperbolic systems. In addition, most analytical results do not state explicitly decay rates and the influence of the source term is assumed to be {\it small}~\cite{L1}. In particular, for kinetic equations with stiff forcing term the latter assumption is in general not fulfilled and it has been shown in \cite{HGY} that a direct extension of the Lyapunov function to a kinetic equation 
will {\em not } yield exponential decay. Therefore, we propose here a {\em novel } Lyapunov function $E_h$ for a particular class of VFP equations. This function has been introduced as energy norm for
 nonlinear Vlasov-Poisson-Fokker-Planck equation on the whole space without stiff source term in \cite{hwang2013vlasov} and we generalize the results to  our setting. 
This function    does not include particular weights, but mixed derivatives and the corresponding decay estimate are used to yield hypocoercivity estimates of the underlying differential operator. 
 The advantage of the novel Lyapunov function is the treatment of the case of small Knudsen numbers and even the  limit of vanishing Knudsen number. We give conditions on boundary feedback stabilization  in terms of the Knudsen number including  the limit.  A particular difficulty for the boundary problem of kinetic equations with small Knudsen number is the treatment of possible boundary layers. The stabilization in the interior of the domain  follows from hypercoercivity estimates. Related work on similar Lyapunov functionals have 
   been used also for results on  Uncertainty Quantification problem for Fokker-Planck related equation with small Knudsen number {\em but } only in the case of  periodic boundary conditions 
   and without stabilizing feedback conditions \cite{ZhuJin18, Zhu18HighDim}. 

This paper is organized as follows. We review the Vlasov-Fokker-Planck equation with feedback control in Section \ref{sec: set up}, and present our main results in Section \ref{sec: main results}. The proof of the main results are given in Section \ref{sec: proof}. The outline of the proof is given in Section \ref{sec: outline}. The proof is divided into three parts. In Section \ref{sec: boundary layer}, we give a sufficient condition such that the feedback control will not yield a boundary layer as the Knudsen number goes to zero. Then we derive an energy estimation for the Lyapunov functional in Section \ref{sec: energy est}. Based on the above two sections, in Section \ref{sec: boundary cond} we give sufficient condition on the boundary such that any perturbation will exponentially decay .\\

\newpage
\begin{gallarynotation}\label{notation}
\textup{
We introduce the following notation frequently used within the manuscript.  Let $t\geq0$ denote the temporal variable, 
 $\O= [0,1]\times\R$  be the phase space of space $x$ and velocity $v,$ respectively. We denote by $H^{-1}_x, L^2_x,  H^1_x, \dots$
 the Lebesgue- and Sobolev spaces in $x$ variable and analogously in $v.$ Further, we define  the following norms on  
 the phase space $\O:$ For suitable integrable functions $g=g(x,v)$ and $\sigma=\sigma(x)$ we 
introduce their corresponding $L^2-$ and higher--order  norms:
\begin{align*} 
 \ll g \rl^2 &:=\| g \|^2_{L^2_{x,v}} =  \int_\O g^2 \,dxdv,  \;  \ll \s \rl^2 := \| \s \|_{L^2_x}^2= \int_0^1 \s^2 dx, \\
 \ll g \rl_\o^2 &:= \ll g \rl^2 + \ll \pt_v g\rl^2 + \ll v g \rl^2, \\ 
 \ll g \rl_V^2 &:= \ll g \rl^2 + \ll \pt_x g\rl^2 \mbox{and}  \ll g \rl_{V,\o}^2 := \ll g \rl_\o^2 + \ll \pt_x g\rl_\o^2.
\end{align*}
We denote by $v \to M(v)$ the global Maxwellian  on $\R$  
\begin{align*} 
    M(v) &:=\frac{1}{\sqrt{2\pi}}e^{-\frac{|v|^2}{2}}.
\end{align*}
The  linearized Fokker--Planck collision operator  is given by  $\mL: H^1_{x,v} \to H^{-1}_{x,v}$,  
\begin{align}\label{def of mL}
    \la \mL g, h \ra= - \int_\Omega M \pt_v\(\frac{g}{\sM}\)  \pt_v\(\frac{h}{\sM}\) dx dv.
\end{align}
The operator $\mL$ satisfies the local coercivity property \cite{duan2010kinetic,ZhuJin18} for $ g=g(x,v) \in H^1_{x,v}$ and $\|g\|_\o<\infty$ and a constant $\lam=\frac14$:
\begin{align}
\label{coercivity}
    -\la \mL g, g \ra \geq \lam\ll (1-\Pi)g \rl^2_\o. \\
    { -\la \mL g, g \ra \geq \lam \(   \| \partial_v   (1-\Pi)g  \|^2 +   \|  v(1-\Pi)g  \|^2 -  \|  (1-\Pi)g  \|^2  \).  }
\end{align}
Here, we denote by $\la \cdot, \cdot \ra$  the $L^2$-scalar product in $x$ and $v$ and  
by $\Pi : L^2 \to \mN(\mL)$ a weighted projection operator onto the null space of $\mL$, i.e., for $g=g(x,v) \in L^2_{x,v}$,
\begin{align}\label{projection}
    \Pi g := \( \int_{\R} g\sM \,dv\) \sM.
\end{align}
For all $\s \in H^1_x$, by the Poincare inequality, there always exists constant $C\geq1$ such that
\begin{equation*}
    C\ll \pt_x\s \rl^2 \geq \ll \s - \int_0^1 \s dx \rl^2.
\end{equation*}
If one adds $\ll \pt_x\s \rl^2$ to both sides of the above inequality, then one has
\begin{equation}
    (C+1)\ll \pt_x\s \rl^2 \geq \ll \s - \int_0^1 \s dx \rl^2_V \implies \ll \pt_x\s \rl^2 \geq \frac{C_s}{2} \ll \s - \int_0^1 \s dx \rl^2_V,
    \label{sobolev const}
\end{equation}
for a constant $0 < C_s \leq 1$. \\
We define by  $h(t,v,x)$ the weighted kinetic distribution function 
\begin{equation}\label{eq: def h}
    h(t,v,x) := \frac{f(t,v,x)}{\sqrt{M(v)}}.
\end{equation}
Denote by $A(t), B(t), \Ax(t),\Bx(t)$  non-negative functions of time $t$ related to the boundary control of $h=h(t,x,v)$ at $x \in \{0,1\}$ 
\begin{equation}
    \label{def of ABCD}
    \begin{aligned}
    &\A(t)  := -\int_{-\infty}^0 \frac{v}{2} h^2(t,0,v)  dv, \qd \B(t) := \int_0^\infty  \frac{v}{2} h^2(t,1,v) dv,\\
    &\Ax(t) := -\int_{-\infty}^0 \frac{v}{2} (\pt_xh)^2(t,0,v)  dv,\qd \Bx(t) := \int^{\infty}_0 \frac{v}{2} (\pt_xh)^2(t,1,v)  dv.
    \end{aligned}
\end{equation}
The function $C_B(t)$  is defined by
\begin{equation}
    \label{def of C_B}
    \begin{aligned}
    &C_B(t) := \frac{\(\sqrt{A(t)} -\sqrt{B(t)} \)^2+\(\sqrt{\Ax(t)} -\sqrt{\Bx(t)} \)^2}{2(\sqrt{\A(t)} + \sqrt{\B(t)})(\sqrt{\Ax(t)} + \sqrt{\Bx(t)})}. 
    \end{aligned}
\end{equation}
The energy $E_h(t)$ at time $t$ and for some positive small parameter $\epsilon>0$ is defined by 
\begin{equation}
\label{eq: def E_h}
E_h(t) := \frac{1}{2}\ll h(t,\cdot,\cdot )\rl^2_V + \e a\la u(t,\cdot), \pt_x\s(t,\cdot) \ra,
\end{equation}
for any constant $a>0$, suitable functions $h,u,\s$  and where $\la \cdot,\cdot \ra$ denotes the $L^2-$scalar product in $x.$ 
We will establish  the estimate 
 \begin{align}
\lv\la u, \pt_x\s \ra\rv \leq \frac{1}{2} \(\ll u \rl^2 + \ll \pt_x\s\rl^2\) \leq \frac{1}{2} \| h(t,\cdot,\cdot)\| ^2_V.
 \end{align} 
Further, we establish that for each fixed time $t$, 
  $E_h(t)$ and $\ll h(t,\cdot,\cdot) \rl_V$ are two equivalent norms. This follows from the  inequalities
\begin{equation}
\label{eq: equiv norm}
\frac{1-a\e}{2}\ll h(t,\cdot,\cdot) \rl^2_V \leq E_h(t) \leq \frac{1+a\e}{2}\ll h(t,\cdot,\cdot) \rl^2_V.
\end{equation}
}
\end{gallarynotation}

\section{The linear Vlasov-Fokker-Planck (VFP) equation and the main result.}\label{sec2}

We are interested in boundary feeback stabilization of  equation \eqref{VFP}   in the sense of Definition \ref{def1} and its associated macroscopic
 limit at $\e \to 0$.
  Hence, we consider the case where the  dynamics given by \eqref{VFP} 
is stabilized at equilibrium $f^*=f^*(x,v)$ by suitable boundary controls \eqref{eq: general cond_f}.
 In order to derive  stabilization results we consider a  linear perturbation $\tilde{f}= f+ f^*$  of an equilibrium state $f^*$, i.e., we assume that $\partial_t f^*(t,\cdot,\cdot)=0.$
 
  Note that  $\tilde{f}$ is supposed to be a kinetic probability density and 
 therefore $$\int_{\O} \tilde{f} (t,x,v) dx dv = 1$$ for all $t\geq 0$ holds true.
This implies that perturbation $f$ is required to fulfill  $\int f(t,x,v) dxdv=0$.  Also, the
  perturbation $f$ fulfills  \eqref{VFP}-\eqref{eq: general cond_f}  due to the linearity of the Fokker--Planck operator and  boundary conditions. 
 
 Hence, we will focus our discussion  on the stabilization of \eqref{VFP} for a perturbation denoted by  $f=f(t,x,v)$. Note that the same model \eqref{VFP} also appears 
as formal first--order approximation to nonlinear VFP equations. However, we do not discuss the stabilization of nonlinear VFP equations in the following. 

\subsection{Definition of the problem.}
\label{sec: set up}
Consider the  linear Vlasov-Fokker-Planck equation for $f=f(t,x,v)$ 
\begin{align}
\e\pt_t f + v\pt_xf - E\pt_v f = \frac{1}{\e}\mF f, \qd (x,v) \in \O,
\label{VFP}
\end{align}
and initial data $f(0,x,v) = f_0(x,v)$. The linear Fokker Planck operator  $\mF$  is given by  
\begin{equation*}
    \mF f = \pt_v\(M\pt_v\(\frac{f}{M}\)\),
\end{equation*} 
where $M$ is the global Maxwellian.  The given bounded function $(t,x) \to E(t,x)$
 models the electric field. The mean free path (or Knudsen number) is $0 < \e\leq 1$. 
The equation is
accompanied by (feedback) boundary conditions specified by the matrix $K(\e)$ \eqref{def of K}: 
\begin{equation}
\label{eq: general cond_f}
\begin{cases}
    f(t,0,v)   = \koo(\e) f(t, 0, -v) +\klo(\e) f(t, 1, v), \qd v>0,  \\
    f(t, 1, v) = \kol(\e) f(t,0,v) + \kll(\e) f(t,1,-v)     ,\qd v<0, 
\end{cases}
\end{equation}
The  matrix $K(\e)$ is given by 
\begin{equation}
\label{def of K}
    K(\epsilon) = \begin{bmatrix}\koo(\e) \qd \klo (\e)   
    \\\kol(\e) \qd \kll(\e) \end{bmatrix}.
\end{equation}
We denote the limit of the coefficients  $K^0$ as $\epsilon$ tends to zero by  
\begin{equation}
    \label{limit}
    K^0 :=  \begin{bmatrix}\koo^0 \qd \klo^0   
    \\\kol^0\qd \kll^0 \end{bmatrix}: = \lim_{\e\to0}\begin{bmatrix}\koo(\e) \qd \klo (\e)  
    \\\kol(\e) \qd \kll(\e) \end{bmatrix}.
\end{equation}

The existence of the solution to the VFP equation in the whole space has been studied for example in \cite{degond1986global, mellet2007global}. However, results on bounded domains do not exist so far. 
	 In this paper, we assume there exists a  solution to the VFP equation \eqref{VFP}, \eqref{eq: general cond_f} 
	 in distributional sense and we assume in the following that $f$ has at least regularity   $f(t,\cdot,\cdot) \in H^1_{x,v}$.

\begin{definition} \label{def1} For fixed $\e>0$, we call the system \eqref{VFP}, \eqref{eq: general cond_f} exponentially stabilizable at equlibrium $f^*\equiv 0$,  if there exists a matrix $K(\e)$  
for any initial data $f_0 \in V$, $\int_\O f_0 dxdv=0$ and  $\| f_0 \|_V < \infty$, such that 
\begin{align*} \| f(t,\cdot,\cdot) \|_V \leq C_0 \| f_0 \|_V \exp( - C_1 t), \; t \geq 0,  \end{align*} 
for some non--negative constants $C_0$ and $C_1.$
\end{definition}
We briefly comment on the notion of stabilization that is an extension of the $L^2-$stabilization for hyperbolic balance laws \cite{L1}.  
The boundary condition \eqref{eq: general cond_f} is a state feedback determined by $K(\e).$  The equilibrium $f^*=0$ is a solution to the stationary Vlasov--Fokker-Planck equation for all $\epsilon$ and stabilization 
 ensures that $f \to f^*$ exponentially fast in time in the norm $\| \cdot \|_V.$ The energy $E_h$ defined by equation \eqref{eq: def E_h} plays
the role of the Lyapunov function in \cite{L1}. Note that contrary to the Lyapunov function,  {\bf no} exponential weights are required here. The dissipation is due to the 
hypercoercivity of the Fokker--Planck operators as shown in the main result of Theorem \ref{thm: general}. 
As indicated in the definition the choice of $K$ may depend on the value of $\e$ and so may  $C_0$ and $C_1,$ too.  An interesting question we discuss below is if the decay rate $C_1$ deteriorates as $\e$ becomes
smaller. We show that in fact, $C_1$ can be chosen {\it independently} of $\e$ provided that the coefficients
of $K(\e)$ fulfill additional conditions, see Section \ref{sec:3.2}.
\par  
In the following section we establish conditions on the operators and prove the  stabilization property in Theorem \ref{thm: general}. For the analysis it
turns out to be advantageous to discuss the properties of equation \eqref{VFP} in terms of the weighted kinetic distribution \eqref{eq: def h}, i.e.,
\begin{align} 
h(t,x,v)=\frac{f(t,x,v)}{\sqrt{M(v)}}, \qd\forall (x,v) \in \Omega, t \geq0.
\end{align} 
Provided that $f$ fulfills  \eqref{VFP} in distributional sense, a simple formal computation shows that $h=h(t,x,v)$ fulfills in this sense
\begin{align}
    \e\pt_th + v\pt_xh - \frac{1}{\e} \mL h = E\(\pt_v - \frac{v}{2}\)h, \qd (x,v) \in \O, t \geq 0.
    \label{micro eq}
\end{align}
The initial conditions and boundary conditions are 
\begin{align}\label{h IC}
h(0,x,v) = \frac{f_0(x,v)}{ \sqrt{M(v)}} \;  \forall (x,v) \in \Omega
\end{align}
and for all $t\geq0$ and for all $\epsilon$
\begin{equation}
\label{eq: general cond}
\begin{cases}
h(t,0,v)   = \koo(\e) h(t, 0, -v) +\klo(\e) h(t, 1, v), \qd v>0;  \\
h(t, 1, v) = \kol(\e) h(t,0,v) + \kll(\e) h(t,1,-v)     ,\qd v<0, 
\end{cases}
\end{equation}
respectively. The operator $\mL$  is defined by equation \eqref{def of mL}.  The set of equations \eqref{micro eq} -- \eqref{eq: general cond} are referred to the {\it{microscopic equations}} or kinetic equations,  since they describe the evolution of the  density function $h$.  For given function $h(t,x,v)$ the  density $\s=\s(t,x)$ and  flux $u=u(t,x)$ are
defined by  
\begin{align}
    \s(t,x) = \int_\R h(t,x,v)\sM \, dv = \int_\R f(t,x,v) dv, \\
     u(t,x) = \int_\R h(t,x,v) v\sM \, dv = \int_\R v f(t,x,v) dv  \label{eq: def s}.
\end{align}
The {\it macroscopic equations} describe the evolution of the density and flux.
Upon multiplication of the microscopic equation by $\sM$ and $v\sM$  and integration on $v$, we formally obtain macroscopic equations as evolution for $\s$ and $u$ respectively. In particular, we have
\begin{align}
\e\pt_t\s + \pt_xu = 0,
\label{continuity}
\end{align}
and for the flux we obtain 
\begin{align}
\e\pt_tu + \pt_x\s + \int v^2\sM (1-\Pi)\pt_xh dv + \frac{1}{\e}u = -E\s. 
\label{macro eq}
\end{align}
The latter equation is obtain by the following computation 
\begin{align}
0= &\int_\R \e \partial_t h v \sqrt{M} + v^2 \sqrt{M} \partial_x h  dv + \frac{1}{\e}  
  \int_\R M \partial_v( v )  \partial_v \left(\frac{h}{\sqrt{M}} \right) dv 
   - E \int_\R v \partial_v ( h \sqrt{M}) dv \nonumber\\
 = & \e \partial_t  u + \int v^2 \sqrt{M} \partial_x h  dv + \int_\R \frac{1}{\e}  \frac{h}{ \sqrt{M}} \partial_v  M dv +
 E \int_\R h \sqrt{M} dv \nonumber\\
   =&  \e \partial_t  u +   \int v^2 \sqrt{M} \partial_x h  dv + \frac{1}{\e} \int_\R  h \sqrt{M} v dv 
   + E \sigma \nonumber\\
   = & \e \partial_t  u +   \int v^2 \sqrt{M} \partial_x (1-\Pi)h  + \partial_x \sigma + \frac{1}{\e} u + E \sigma. 
\end{align}
The last line follows due to the definition of $\Pi$ in equation \eqref{projection}, i.e., 
\begin{align}
\Pi h = \int h \sqrt{M} dv \sqrt{M} = \sigma \sqrt{M}
\end{align}
and
\begin{align}
\int_\R v^2 \sqrt{M} \partial_x \sigma \sqrt{M} dv = \partial_x \sigma.
\end{align}
Since by  definition of $\Pi$ and $\s$ we obtain that  the operators $\Pi$ and $1-\Pi$ are perpendicular to each other under $L^2_{x,v}$, that is,
\begin{align}
    \ll h\rl^2 = \ll \Pi h \rl^2 + \ll (1-\Pi) h\rl^2 = \ll \s \rl^2 + \ll (1-\Pi) h\rl^2.
\end{align}
The first equality holds true due to 
\begin{align}
\int_\Omega (\Pi h)(1-\Pi)h dx dv = \int_\Omega \sigma \sqrt{M} ( h - \sigma \sqrt{M} ) dxdv = 0, \\
\int_\Omega (\Pi h) dxdv=\int \sigma^2 \(\int_\R M dv\) dx = \ll \sigma \rl.
\end{align}

\subsection{Theoretical  results.}
\label{sec: main results}
In this section we present conditions on the matrix $K(\epsilon)$ and the external field  $E$, such that the dynamics given by equation \eqref{VFP} is stabilizable in the sense of Definition \ref{def1}. 

For the external electric field we assume sufficient regularity and growth conditions: 
\begin{condition}
	\label{cond: E}
	The electric field $E: \R^+_0 \times [0,1]\to \R$ is assumed to be sufficiently smooth, for all time $t \in \R^+_0$, and its derivatives are  bounded by the following constants independent of $t$:
	\begin{equation}
	\label{bdd of E}
	\ll E(t,\cdot) \rl_{L^\infty_x}, \ll \pt_x E(t,\cdot) \rl_{L^\infty_x}, \ll \pt_x^3 E(t,\cdot) \rl_{L^\infty_x} \leq \frac{C_E}{2} \leq \frac{\lam C_s}{16},
	\end{equation}
	where $\lam, C_s$ are constants defined in \eqref{coercivity}, \eqref{sobolev const}. In addition, at the boundary $x\in \{0,1\}$, we assume that the electric field and its second derivative are periodic:
	\begin{equation}
	\label{bdy of E}
	E(t,0) = E(t,1) = 0, \qd \pt_x^2E(t,0) = \pt_x^2E(t,1).  
	\end{equation}
\end{condition}

As discussed in Section \ref{sec2}  the perturbations $f$ and $f_0$ of the steady state are required to have zero mean. 
\begin{condition}
\label{ass initial}
	Assume for initial data $f_0(\cdot,\cdot) \in V$,
	\begin{equation}
	\int f_0(x,v) \,dx\,dv = 0.	\label{ass: zero mass}	
	\end{equation}
\end{condition}

Next, we present exponential stability results in the case of small and large electric field governed by $C_E$ 
as well as  in the limit case $\e=0.$ The latter corresponds to stabilization of the formal hydrodynamic limit. 
We refer to the remarks below the theorems for some discussion as well as to Lemma \ref{lem01} for details
on the imposed assumptions.  The next Theorem states the exponential stability for solutions with small electric field. 
\begin{theorem}
	\label{thm: general}
	Assumptions \ref{cond: E} and \ref{ass initial} hold true and let $\e >0$.  Further,  assume that  the upper bound of the electric field $C_E$ given by  equation \eqref{bdd of E} 
	and the coefficients of the matrix $K(\epsilon)$ given by equation \eqref{def of K} satisfy 
	\begin{equation} \label{cond_mh1} 
	\begin{aligned}
	&C_E  < \frac{3C_s}{2}-\frac{3C_s}{2+ C_B(t)},\;  \\
	&0\leq \kll(\epsilon) = \koo(\epsilon) \leq 1, \; 
	\kol(\e) = \klo(\e) = 1-\koo(\e),
	\end{aligned}
	\end{equation}
	
Then, any weak solution $f\in C^0(\R_+^0; H^1_{x,v})$   to (\ref{VFP}) with boundary condition \eqref{eq: general cond_f}  will decay exponentially in time to zero with rate given by
	\begin{equation}
	\label{eq: general_decay}
	\ll \frac{f(t)}{\sM} \rl_V^2 \leq \frac54\ll \frac{f(0)}{\sM}\rl_V^2e^{-2\xi t},
	\end{equation}
	where 
	\begin{equation}
	\label{eq: def of xi}
	\begin{aligned}
	&\xi = \min\l\{ \frac{\lam-C_E-4a}{\e^2},  \frac{a(3C_s-2C_E)-4C_E}{8} \r\}>0.
	\end{aligned}
	\end{equation}
	Here,  $\frac{4C_E}{3C_s-2C_E} < a < \min\l\{C_B(t), \frac{\lam - C_E}{4}\r\}$ and the functions $C_B(t), C_s, \lam$ are defined in equations (\ref{def of C_B}), (\ref{sobolev const}) 
	and (\ref{coercivity}), respectively.	
\end{theorem}


 If the electric field is large, we  observe that we have  less degrees of freedom to chose a linear feedback boundary condition.  The following theorem states sufficient conditions on the feedback matrix $K(\e)$ for  exponential stability and possibly large electric fields.
\begin{theorem} 
		\label{thm: general limit}
	Assume Assumptions \ref{cond: E} and \ref{ass initial} hold true. Further,   assume that  the upper bound of the electric field $C_E$ given by  equation \eqref{bdd of E} 
	and the coefficients of the matrix $K(\epsilon)$ given by equation \eqref{def of K} satisfies, 
	\begin{equation}\label{cond_mh2}
C_E \geq \frac{3C_s}{2}-\frac{3C_s}{2+ C_B(t)}, \;  \kll(\e)= \koo(\e) = 0, \; \kol(\e) = \klo(\e) = 1.
\end{equation}
Then, any weak solution $f\in C^0(\R_+^0; H^1_{x,v})$   to (\ref{VFP}) with boundary condition \eqref{eq: general cond_f}  will decay exponentially in time to zero with rate given by
\begin{equation}
\label{eq: general_decay 2}
\ll \frac{f(t)}{\sM} \rl_V^2 \leq \frac54\ll \frac{f(0)}{\sM}\rl_V^2e^{-2\xi t},
\end{equation}
where for $\e>0$
	\begin{equation*}
\begin{aligned}
&\xi = \min\l\{ \frac{\lam-C_E-4a}{\e^2},  \frac{a(3C_s-2C_E)-4C_E}{8} \r\} > 0,
\end{aligned}
\end{equation*}
where $\frac{4C_E}{3C_s-2C_E} < a < \frac{\lam - C_E}{4}$  and the values $C_s, \lam$ are defined in equations (\ref{sobolev const}) and (\ref{coercivity}), respectively.
\end{theorem} 

The formal hydrodynamic limit is obtained for $\e \to 0.$ In this case both theorems yield the same
 condition on the matrix $lim_{e \to 0} K(\e) = K(0).$ Namely,  in the limit we obtain 
 $$ K(0)=\begin{pmatrix} 0 & 1 \\ 1 & 0 \end{pmatrix}.$$
This feedback matrix describes in fact periodic boundary conditions without additional damping. Hence, 
 exponential decay of solutions is guaranteed independent of $\e$ if periodic conditions are prescribed. However,
 the decay rate deteriorates. Hence, the previous results show that the only asymptotic preserving feedback boundary conditions are the periodic boundary conditions. The following lemma is a consequence of the previous 
 theorems. 
 
 \begin{lemma}
 	\label{thm: general2}
 	Let Assumptions \ref{cond: E} and \ref{ass initial} hold true and let $\e=0.$ Assume that 
 	 $$ K =\begin{pmatrix} 0 & 1 \\ 1 & 0 \end{pmatrix},$$
 	and  $E$ fulfills assumption \eqref{bdd of E}. Then, any weak solution $f\in C^0(\R_+^0; H^1_{x,v})$   to (\ref{VFP}) with boundary condition \eqref{eq: general cond_f}  will decay exponentially in time to zero with rate given by
 	\begin{equation}
 	\ll \frac{f(t)}{\sM} \rl_V^2 \leq \frac54\ll \frac{f(0)}{\sM}\rl_V^2e^{-2\xi t},
 	\end{equation}
 	where 
 	\begin{equation}
 	\xi =  \frac{a(3C_s-2C_E)-4C_E}{8}  >0.
 	\end{equation}
 	Here,  $\frac{4C_E}{3C_s-2C_E} < a < \min\l\{C_B(t), \frac{\lam - C_E}{4}\r\}$ and the functions $C_B(t), C_s, \lam$ are defined in equations (\ref{def of C_B}), (\ref{sobolev const}) 
 	and (\ref{coercivity}), respectively.	
 \end{lemma}

Theorems \eqref{thm: general} and \eqref{thm: general limit} will be proven in the forthcoming section.  Lemma \eqref{thm: general2} is a simple consequence. Prior to proceeding, some remarks are in order.
\begin{itemize}
	\item  The result guarantees stabilization exponentially fast in time provided that the size of the electric field is  suitable small. This case is covered by assumption \eqref{cond_mh1}.  In case of potentially large fields we may only stabilize the dynamics by periodic boundary conditions as given by equation \eqref{cond_mh2}. 
	\item Due to the linearity of the underlying equations we also obtain that the boundary condition \eqref{def of K} translates to a condition on $f^* + f$ where $f^*$ is any  equilibrium state. 
	\item The question whether we obtain a boundary layer for small values of $\e$ will be addressed in the beginning of Section \ref{sec: proof}.
	\item In all cases we have $\xi >0$, which is guaranteed by $\frac{4C_E}{2C_s - 2C_E} < a < \frac{\lam - C_E}{4}$ in both theorems. In Theorem \ref{thm: general limit}, the smallness of $C_E \leq \frac{\lam C_s}{8}$ in Assumption \ref{cond: E} guarantees that there always exists constants $a>0$ in the range of $(\frac{4C_E}{2C_s - 2C_E}, \frac{\lam - C_E}{4})$. In Theorem \ref{thm: general}, we further require $\frac{4C_E}{2C_s - 2C_E}  <a < C_B$, the existence of $a>0$ can be guaranteed by the additional assumption in \eqref{cond_mh1} that $C_E  < \frac{3C_s}{2}-\frac{3C_s}{2+ C_B(t)}$.
	\item Stated in terms of the boundary feedback matrix $K(\e)$ it is clear that condition \eqref{cond_mh2} is a particular case of condition \eqref{cond_mh1} and hence included in Theorem \ref{thm: general}. However, for positive value of $\e$, sufficiently small electric field $C_E$ and under the additional assumption on $0<a \leq C_B(t)$ we have a weaker condition on the boundary feedback and thus a
	wider range of possible feedback matrices $K(\e).$ The additional condition on $a \leq C_B(t)$ in the following lemma needs to be combined with the condition $\frac{4C_E}{3C_s-2C_E} < a < \frac{\lam - C_E}{4}$ of Lemma \ref{lemma: energy est}. This shows that only if $C_E<\frac{3C_s}{2}-\frac{3C_s}{2+ C_B(t)}$, the existence of $a>0$ is granted. Therefore, we have differentiated the results into two cases depending on the size of the electric
	fields $C_E.$ This is detailed in Lemma \eqref{lem01}.
	\item In Section \ref{sec:3.2} we show that the conditions imposed on the feedback $K(\e)$ can not be changed
	substantially provided exponential stability is obtained using the Lyapunov function $E_h$ defined by \eqref{eq: def E_h}.
		\item In the  limit $\e\to0$ in Theorem \ref{thm: general}, we observe that only for periodic boundary conditions exponential stability is obtained. The size $C_E$ of the electric field is independent of this condition. Hence, 
	only periodic boundary conditions are asymptotically preserving for stabilization of the considered Vlasov--Fokker--Planck dynamics. It is still an open question if for other Lyapunov- (or energy) functionals
	exponential stabilization results can be obtained. 
\end{itemize}

\section{Proof of the main results.}
\label{sec: proof}

The proofs of the two theorems are similar and we therefore will focus the proof of Theorem \ref{thm: general}
and point out the modifications for the proof of Theorem \ref{thm: general limit}. 

\subsection{Outline of the proof and  sufficient conditions on $K$ for exponential decay.} 
\label{sec: outline}
Prior to the detailed analysis of the proof, we discuss the dependence on the Knudsen number  $\e$ and  derive  necessary conditions on the boundary control matrix $K=K(\e)$ to establish the exponential decay. The conditions  will be summarized below together with some explanation of their origin. 
\par 
We aim to find  conditions for exponential stability possibly   independent of  the parameter $\e$.  In Lemma \ref{lemma: bdy layer} we formally derive the limiting equation of (\ref{VFP}) for $\e \to 0$ and we state a  sufficient condition on $K$ such that no boundary layer appears. Lemma \ref{lemma: bdy layer} yields the following sufficient condition on $K$. To simplify the notation, we drop
the dependence of the coefficients of $K$ on $\e$ if $\e>0$. The limiting case $K=K(0)$ is indicated in the coefficients as in equation \eqref{def of K} by a superscript zero.
\begin{condition}
	\label{const 1}
	We assume that the entries of $K=K(0)$ fulfill the following quadratic equations,
	\begin{align}
	&\(1-\koo^0\)\(1-\kll^0\) = \klo^0\kol^0, \label{cost1_1}\\ 
	&\(1+\koo^0\)\(1+\kll^0\)= \klo^0\kol^0.\label{cost1_2}
	\end{align}
\end{condition}

The proof of Theorem \ref{thm: general} relies on estimates of the Lyapunov functional $E_h$ defined in (\ref{eq: def E_h}). The following energy estimate is proved in Lemma \ref{lemma: energy est}. 
	\begin{equation}
	\label{eq: energy est 1}
	\begin{aligned}
	\pt_tE_h + \xi \ll h\rl^2_{V}\leq& - \int_{\R} \frac{v}{2} \(h^2(t,1,v) - h^2(t,0,v)\) dv \\
	&- \int_{\R} \frac{v}{2} \((\pt_xh)^2(t,1,v) - (\pt_xh)^2(t,0,v)\) dv \\
	&-a \(u\pt_xu(t,1) - u\pt_xu(t,0)\)+ \frac{aC_s}{2\e^2}\( \int_0^t\(u(s,1) - u(s, 0)\) ds\)^2.
	\end{aligned}
	\end{equation}
In order to obtain  exponential decay of $t \to \ll h(t) \rl^2$,  we in particular require the right--hand side of the above equation to be  non-positive. This in turn requires, in Lemma \ref{lemma: bdy for u}, a  second condition on $K$ making the  last term of the above equation vanish. This condition is as follows:
\begin{condition}
	\label{const 2}
	We assume that the entries of $K$ fulfill the following conditions:
	\begin{equation*}
	\koo+\kol = 1,\qd \klo + \kll = 1.
	\end{equation*}
\end{condition}
Finally, we require, in Lemma \ref{lemma: bdy for u}, the first three terms of the RHS of equation (\ref{eq: energy est 1})  to be non--positive. This leads to the final conditions on $K$:
\begin{condition}
	\label{const 3}
	We assume that the entries of $K$ fulfill the following  condition, where $A,B,\Ax,\Bx$ are functions of $t$ defined in equation (\ref{def of ABCD}):
	\begin{equation*}
	\begin{aligned}
	&-  2\koo(1- \koo)\(\A+\Ax\)-2\kll(1- \kll)\(\B+\Bx\)\\
	&+2\(\lv\kll(1-\koo)\rv + \lv\koo(1-\kll)\rv\)\(\sqrt{\A\B}+\sqrt{\Ax\Bx}\)\\
	&+4a\(\lv 1-\kll \rv\sqrt{\B} + \lv 1-\koo\rv\sqrt{\A}\)\(\lv\koo\rv\sqrt{\Ax}+\lv\kll\rv
	\sqrt{\Bx}\)\leq 0,
	\end{aligned}
	\end{equation*}
\end{condition}
The above assumption is a technique assumption following from the energy estimate (\ref{eq: energy est 1}), so that one could has the exponential decay of $\ll h(t) \rl$ in time.

The next lemma shows that the 
conditions on $K(\e)$ given in Theorem \ref{thm: general} and Theorem \ref{thm: general limit}
fulfill the previous assumptions. Depending on the size of the electric field $C_E$ we can have different
conditions on $K(\e).$

\begin{lemma} \label{lem01}
	Let $a \leq C_B(t)$ for all $t \geq 0$ and assume that the entries of $K(\e)$ are given by equation \eqref{cond_mh1}. 
Then, the assumptions \ref{const 2}, \ref{const 2} and \ref{const 3} are fulfilled.
	\par
	Alternatively, assume the entries of $K(\e)$ are given by equation \eqref{cond_mh2}. 
	Then, the assumptions \ref{const 1}, \ref{const 2} and \ref{const 3} are fulfilled.
	\end{lemma}

	\begin{proof}

In fact, the only nontrivial part is to compute Assumption \ref{const 3}.  In the case
$h=0$ (then $A = B = 0$), the inequality is trivially fulfilled. Since $\koo=\kll$, Assumption \ref{const 3} can be written as
\begin{align*}
&2 \koo(1-\koo) \(  - \A - \A_x - \B - \B_x + 2 \sqrt{ \A \B} + 2\sqrt{ \A_x \B_x } + 2 a \( \sqrt{\B} + \sqrt{\A} \) \(\sqrt{\A_x} +\sqrt{\B_x} \)\) \\
=& 2 \koo (1-\koo) \( - ( \sqrt{\A} -\sqrt{\B} )^2 - ( \sqrt{\A_x}-\sqrt{\B_x} )^2 + 2a 
\( \sqrt{\B} + \sqrt{\A} \) \(\sqrt{\A_x} +\sqrt{\B_x} \)\) \\
=&   4 \koo (1-\koo)   \( \sqrt{\B} + \sqrt{\A} \) \(\sqrt{\A_x} +\sqrt{\B_x} \) 
\(  a   - C_B   \) \leq 0.
\end{align*}
By the above equation and if  we have $a \leq C_B$, then there is no additional assumption on $K$. The condition \eqref{cond_mh1} is due to Assumptions \ref{const 1} and \ref{const 2}.  On the other hand, if $a > C_B$, then the above inequality holds true if $\koo = 1$ or $\koo = 0$. Since the case $\koo = 1$ is a contradiction to Assumption \ref{const 1} for sufficiently small $\e = 0$, hence,  we obtain \eqref{cond_mh1}.
\end{proof}

Summarizing, provide that   $K=K(\e)$ satisfies the assumptions  \ref{const 1} - \ref{const 3}, and assuming   \eqref{bdd of E}, \eqref{bdy of E} and  \eqref{ass: zero mass},  we obtain the following  {\bf uniform} in $\e$ estimate of the energy $E_h$:
\begin{equation}\label{tem1}
\pt_tE_h + \xi \ll h\rl^2_V\leq 0, \;  \forall 0 \leq \e \leq 1,
\end{equation}
where $\xi$ is defined by equation (\ref{eq: def of xi}). The result will be proven below in Lemma \ref{lemma: energy est}. Assume for now that \eqref{tem1} holds. Then, integrating over $t$ and using the equivalence (\ref{eq: equiv norm}), we obtain  an estimate for $\|h(t)\|_V^2$ as 
\begin{align}
\frac{1-a\e}{2} \ll h(t) \rl^2_V \leq  E_h(t) \leq  E_h(0) - \xi \int_0^t\ll h(s)\rl^2_V ds,\\
\leq  \frac{1+a\e}{2} \ll h(0) \rl^2_V  - \xi \int_0^t\ll h(s)\rl^2_V ds. 
\end{align}
This is a Gronwall inequality for the norm of $h$ and therefore, 
\begin{align}\label{tem2}
 \ll h(t) \rl^2_V \leq & \frac{1+a\e}{1-a\e} \ll h(0) \rl^2_V  e^{-\frac{2\xi}{1-a\e}t}.
\end{align}
In Lemma \ref{lemma: energy est} we establish that 
$4 a \leq (\lam-C_E) $. 
Since $C_E \geq 0$ and $\lam \leq 1,$ see Section \ref{notation}, we have $ a \leq \frac14.$  Since we are interested in small Knudsen numbers, we may assume $\e \leq 1$ and obtain 
\begin{align} 0 \leq a \e \leq \frac14. \end{align}
Hence, we obtain a {\bf uniform} bound on $\ll h(t) \rl^2$ as
\begin{equation}
\label{eq: exponential decay}
\begin{aligned}
\ll h(t) \rl^2_V \leq & \frac{5}{4} \ll h(0) \rl^2_V  e^{-2\xi t}.
\end{aligned}
\end{equation}
Due to the definition of $h$ in equation \eqref{eq: def h} the same decay rate is obtained for $f.$

\subsection{ Alternative conditions on $K(\e)$. } \label{sec:3.2}

Clearly, the particular choice of $K(\e)$ in Theorem \ref{thm: general} and Theorem \ref{thm: general limit} 
are not the only possibility to fulfill the assumptions \ref{const 1} - \ref{const 3}. In the following section we give alternative  conditions on $K(\e)$ such that previously introduced assumptions \ref{const 1} - \ref{const 3} hold true. For readability reasons we drop now the dependence of the coefficients of $K(\e)$ on $\e$ whenever there is no confusion. We start by simplifying assumption \ref{const 3} in the case that $\koo, \kll \leq 1.$ 

Provided that 
\begin{equation}
\label{const3_1}
0\leq \koo,\kll \leq 1
\end{equation}
holds true, then assumption \ref{const 3} holds true  provided that 
\begin{equation}
\label{eq: const3_2}
\begin{aligned}
&-2\(\koo\sqrt{A} -\kll\sqrt{B} \)\((1-\koo)\sqrt{A} -(1-\kll)\sqrt{B}\)\\
&-2\(\koo\sqrt{\Ax} -\kll\sqrt{\Bx} \)\((1-\koo)\sqrt{\Ax} -(1-\kll)\sqrt{\Bx}\)\\
&+4a\((1-\koo)\sqrt{\A} +(1-\kll)\sqrt{\B}\) \(\koo\sqrt{\Ax} +\kll\sqrt{\Bx} \)\leq 0.
\end{aligned}
\end{equation}
 Even though the equation \eqref{eq: const3_2} seems more complicated,  it allows for a refined
analysis in the following two cases summarized in the following corollary.


\begin{corollary} \label{mh cor1}
	Assume that the matrix $K(\e)\in\R^{2 \times 2}$ defined by equation \eqref{def of K}  fulfills
	 \eqref{const3_1}.
	Assume that  for all $\e \geq 0$,  
	\begin{equation}\label{mh temp11}
	\begin{aligned}
	&  \koo(\e) = \kll(\e). 
	\end{aligned}
	\end{equation}
	Then, assumption \eqref{eq: const3_2} is fulfilled provided that 
\begin{equation}
		\label{const3_4}
a \leq C_B(t)= \frac{\(\sqrt{A} -\sqrt{B} \)^2+\(\sqrt{\Ax} -\sqrt{\Bx} \)^2}{2(\sqrt{\A} + \sqrt{\B})(\sqrt{\Ax} + \sqrt{\Bx})}.
		\end{equation}
\end{corollary}
In fact, equation (\ref{eq: const3_2}) reads under assumption \eqref{mh temp11} for $\koo=\koo(\e)$
\begin{align}\label{final eq 2}
2\koo(1-\koo)\l[-(\sqrt{\A} - \sqrt{\B})^2-(\sqrt{\Ax} - \sqrt{\Bx})^2 + 2a(\sqrt{\A} + \sqrt{\B})(\sqrt{\Ax} + \sqrt{\Bx})\r]\leq 0.
\end{align}
Since $\koo(1-\koo) \geq 0$ according to (\ref{const3_1}), the above inequality holds true provided that
either $\koo(1-\koo) =0$ or 
\begin{equation}
\label{const3_5}
\begin{aligned}
a \leq C_B(t), \qd C_B(t) = \frac{\(\sqrt{A} -\sqrt{B} \)^2+\(\sqrt{\Ax} -\sqrt{\Bx} \)^2}{2(\sqrt{\A} + \sqrt{\B})(\sqrt{\Ax} + \sqrt{\Bx})}.
\end{aligned}
\end{equation}

\par 
In the subsequent  Lemma \ref{lemma: energy est} 
 we obtain the following bounds on the constants $a$ and $C_E:$  
\begin{equation}
\frac{4C_E}{3C_s-2C_E} < a < \frac{\lam - C_E}{4},\qd C_E \leq \frac{\lam C_s}{8}. \label{cost a 1}
\end{equation}
Furthermore, note that $0 \leq  \lam \leq 1.$ If we combine the inequalities \eqref{cost a 1}
and  estimate \eqref{const3_5} we obtain 
\begin{align}\label{final cond 1}
C_B \geq \frac{4 C_E}{ 3 C_s - 2 C_E} \implies C_E < \min\left\{ \frac{\lam C_s}8, 
\frac32 C_s - \frac{3C_s}{2+C_B} \right \}.
\end{align}



In the case condition \eqref{final cond 1} does not hold we have the
following result which follows directly by equation \eqref{final eq 2}. 
In this case the  boundary conditions are either periodic or reflective.
\begin{corollary} \label{mh cor2}
	For the matrix $K(\e)\in\R^{2 \times 2}$ defined by equation \eqref{def of K}, assume further that 	\begin{equation}\label{mh temp12}
	\begin{aligned}
	&  \koo(\e) = \kll(\e) \mbox{ and } \koo(\e) \in \{ 0, \; 1 \}. 
	\end{aligned}
	\end{equation}
	Then, assumption \eqref{eq: const3_2} is fulfilled. 
\end{corollary}


%
%

Finally, we have the following remark on the case $h=0.$ In this case $A=B=0$ and therefore \eqref{const3_4} is not defined. However, the 
state $h=0$ corresponds to zero energy $E_h(t)=0$ and it is precisely the 
state that we stabilize. Therefore, as long as $h \not =0$, we may assume 
$A,B > 0.$

\subsection{Limiting equation and boundary layer.}
\label{sec: boundary layer}
The particular case of Theorem \ref{thm: general} requires an analysis of the spatial derivative of $\partial_x h.$ A formal differentiation of the boundary conditions \eqref{eq: general cond} to the VFP equation \eqref{micro eq} leads to the following result.  

\begin{lemma}
	\label{lemma: bdy pt_x}
	Let $h$ be a sufficient smooth solution to  systems \eqref{micro eq},   \eqref{h IC} and \eqref{eq: general cond} . Then, $\pt_xh$ fulfills the following boundary conditions. 
	\begin{equation}
	\label{eq: general cond pt_xh}
	\begin{cases}
	\pt_xh(t,0,v)   = -\koo \pt_xh(t, 0, -v) +\klo \pt_xh(t, 1, v), \qd v>0;  \\
	\pt_xh(t, 1, v) = \kol \pt_xh(t,0,v) - \kll \pt_xh(t,1,-v)     ,\qd v<0. 
	\end{cases}
	\end{equation}
\end{lemma}
\begin{proof}
	Differentiating the boundary condition \eqref{eq: general cond} with respect to time $t$ and with respect to 
	velocity $v$ we obtain 
	the relations
	\begin{equation}
	\label{pf: relation}
	\begin{aligned}
	&\pt_th(t,0,v)   = \koo \pt_th(t, 0, -v) +\klo \pt_th(t, 1, v), \qd v>0,\\
	&\pt_v h(t,0,v)   = - \koo \pt_v h(t, 0, -v) +\klo \pt_v h(t, 1, v), \qd v>0, \\
	& v h(t,0,v) =  - \koo v h(t, 0, -v) +\klo v  h(t, 1, v), \qd v>0,\\
	&\pt_v^2h(t,0,v)   = \koo \pt_v^2h(t, 0, -v) +\klo \pt_v^2h(t, 1, v), \qd v>0.
	\end{aligned}
	\end{equation}
	The strong form of the linearized Fokker--Planck operator $\mL$ is given by 
	\begin{align}
	\mL h = \frac{1}{\sqrt{M}}    \partial_{v}  \( M \partial_v \frac{ h }{ \sqrt{M}}\).
	\end{align}
	If we differentiate $h$ we obtain 
	\begin{align}
	\mL h =\frac12h  - \frac{v^2}4h + \partial_v^2 h.
	\end{align}
Evaluating the VFP equation \eqref{micro eq} at $x=0$ and assumption \eqref{bdy of E} that $E = 0$ at boundary yields 	
	\begin{align}
	\e\pt_th(t,0,v) + v\pt_xh(t,0,v) - \frac{1}{\e}\(\frac12h  - \frac{v^2}4h + \pt_v^2h \)(t,0,v)  = 0
	,\qd v>0.\label{pf: x=0}
	\end{align}
	Reformulating the equation in terms of $v \partial_x h(t,0,v)$ 
we obtain 
%
%
	\begin{align}
	&v\pt_xh(t,0,v) = \l[-\e\pt_th + \frac{1}{\e}\( \frac12h  - \frac{v^2}4h  + \pt_v^2h  \)\r](t,0,v) \\
	=&\koo\l[-\e\pt_th + \frac{1}{\e}\( \frac12h  - \frac{v^2}4h + \pt_v^2h  \)\r](t,0,-v) + \\ 
	& + \klo\l[-\e\pt_th + \frac{1}{\e}\(\frac12h  - \frac{v^2}4h + \pt_v^2h  \)\r](t,1,v)\\
	=&-\koo v\pt_xh(t,0,-v) +\klo v\pt_xh(t,1,v),
	\end{align}
	where the last line follows by replacing $v$ by $-v$ for $v< 0$ in equation \eqref{pf: x=0}, i.e., for $v<0$ 
		\begin{align}
	\e\pt_th(t,0,-v)  - v\pt_xh(t,0,-v) - \frac{1}{\e}\( \frac12h  - \frac{v^2}4h   + \pt_v^2h \)(t,0,-v)  = 0\\
\implies	- v\pt_xh(t,0,-v) =  	\( - \e\pt_th   + \frac{1}{\e}\( \frac12h  - \frac{v^2}4h + \pt_v^2h \)  \)(t,0,-v)  
	\end{align}
	
	A similar computation can be performed for the boundary condition at $x=1.$ 
	\end{proof}

In the following lemma we discuss the limiting equation for vanishing Knudsen numbers and
 solutions $f$ to the original VFP equation \eqref{VFP}. 
\begin{lemma}
	\label{lemma: bdy layer}
For fixed initial data $f_0(x,v)=\sigma_0(x) M(v) \in V$ for some given function $\sigma_0(x)$ and fixed $\e>0$ we assume there exist a sufficiently smooth solution 
to equations \eqref{VFP} and \eqref{eq: general cond_f}. The solution is denoted by $f^\e=f^\e(t,x,v)$. We assume that a sufficiently smooth function $f^0(t,x,v)$ exists as limit of the sequence of solutions $f^\e$ for $\e \to 0.$  The following system of differential equation is then formally fulfilled by $f^0(t,x,v).$   
The solution $f^0(t,x,v)$ is given by 
\begin{align}
f^0(t,x,v)=\sigma^0(t,x) M(v)
\end{align}
where  $\s^0$ satisfies for all $x \in [0,1], t \geq 0, $
	\begin{equation}\label{macro}
	\pt_t\s^0 - \pt_x\(\pt_x\s^0+E\s^0\) =0.
	\end{equation}
	The initial condition is given by $\sigma(0,x)=\sigma_0(x).$ 
Only if
	\begin{equation}
	\label{constraint3}
	\(1-\koo^0\)\(1-\kll^0\) = \klo^0\kol^0, \qd \(1+\koo^0\)\(1+\kll^0\)= \klo^0\kol^0,
	\end{equation}
then	the function $\s^0 $ fulfills the boundary condition 
	\begin{equation*}
	\s^0(t,0) = \klo\s^0(t,1), \qd \pt_x\s^0(t,0) = \klo\pt_x\s^0(t,1).
	\end{equation*}
\end{lemma}
The condition \eqref{constraint3} shows that in the zero Knudsen number limit the limiting system \eqref{macro} and \eqref{constraint3} determine the solution $f^0(t,x,v).$ Hence, $f^0$ is defined
up to the boundary by the macroscopic quantity $\sigma(t,x)$ and the global Maxwellian. We therefore do {\bf not} observe a boundary layer as transition phase between the kinetic distribution at the boundary and the small Knudsen limit in the interior. 

\begin{proof}
	The proof is given by a formal expansion of $f$ in a Hilbert series in terms of the parameter $\e.$ 
	Assume   $f = h\sM$ can be expanded in terms of the Knudsen number  $f = f_0 + \e f_1 + \e^2f_2+O(\e^3)$.  
	Recalling equation \eqref{VFP} we have 
	\begin{align}
	\e\pt_t f + v\pt_xf - E\pt_v f = \frac{1}{\e}\mF f, \qd (x,v) \in \O,
	\label{VFP2}
	\end{align}
	we obtain for the expansion terms of the following order in $\e$
	\begin{equation}
	\label{eq: order}
	\begin{aligned}
	&O\(\frac{1}{\e}\): \qd \mF f_0 = 0,\\
	&O(1): \qd v\pt_xf_0 -E\pt_vf_0 = \mF(f_1),\\
	&O(\e): \qd \pt_tf_0 + v\pt_xf_1 -E\pt_vf_1 = \mF(f_2).
	\end{aligned}
	\end{equation}
If  $f_0$ has the following form 
	\begin{equation}
	\label{eq: f_0}
	f_0 =  \s^0(t,x)M(v), \; (x,v) \in \Omega
	\end{equation}
	 we obtain $\mF f_0 =0.$ We define $\s^0(t,x):=\int_\R f_0(t,x,v) dv$ for all $x \in [0,1]$ and $t \geq 0.$ 
		
	
	In the following computations we derive an equation for $\s^0.$ For example, the specific form of $f^0$ given by equation  \eqref{eq: f_0} leads in equation \eqref{eq: order} to 
	\begin{equation} 
	\begin{aligned}
	v\pt_x\s_0M +vE\s_0M =& \pt_v\(M\pt_v\(\frac{f_1}{M}\)\).
	\end{aligned}
	\end{equation}
	Multiplying $v$ and integrating over $v$ yields 
	\begin{equation}
	\label{eq: f_1}
	\int vf_1dv =-\( \pt_x\s^0 + E\s^0\).
	\end{equation}
Integration of the $O(\e)$ term in equation \eqref{eq: order} and using the previous relation \eqref{eq: f_1}
the closed formulation \eqref{macro} in terms of $s^0$ is obtained:
	\begin{equation}
	\pt_t\s^0 - \pt_x\(\pt_x\s^0+E\s^0\) =0.
	\end{equation}
	We define boundary condition for both $\s^0$ and $\pt_x\s^0$. Inserting $\lim_{\e\to0} f^\e = \s^0M$ into (\ref{eq: general cond_f}) and using that the global Maxwellian is symmetric, $M(v) = M(-v)$, we obtain 
	\begin{equation*}
	\begin{cases}
	\s^0(t,0)  = \koo\s^0(t,0) + \klo\s^0(t,1),\\
	\s^0(t,1)  = \kol\s^0(t,0) + \kll\s^0(t,1),
	\end{cases}
	\end{equation*}
	which is equivalent to 
	\begin{equation*}
	\begin{cases}
	\(1-\koo^0\)\s^0(t,0) = \klo^0\s^0(t,1),\\
	\(1-\kll^0\)\s^0(t,1)  = \kol^0\s^0(t,0).
	\end{cases}
	\end{equation*}
	The above equation holds for any $\s^0(t,0), \s^0(t,1)$ if and only if 
	\begin{equation*}
	\(1-\koo^0\)\(1-\kll^0\) = \klo^0\kol^0.
	\end{equation*}
	For boundary conditions $\pt_x \s^0$ we repeat the previous computation for the 
	 $\lim_{\e\to0}\pt_xf^\e = \pt_x\s^0 M$ and obtain 
	\begin{equation*}
	\begin{cases}
	\pt_x\s^0(t,0)  =-\koo\pt_x\s^0(t,0) + \klo\pt_x\s^0(t,1),\\
	\pt_x\s^0(t,1)  = \kol\pt_x\s^0(t,0) -\kll\pt_x\s^0(t,1). 
	\end{cases}
	\end{equation*}
This holds true only if 
	\begin{equation*}
	\(1+\koo^0\)\(1+\kll^0\)= \klo^0\kol^0,
	\end{equation*}
	which finishes the formal proof. 
\end{proof}

\subsection{Estimates on the Lyapunov Function.}
\label{sec: energy est}
In this section we establish estimates on the Lyapunov function $E_h$ defined by equation 
\eqref{eq: def E_h} and recalled here for convenience: 
\begin{align*}
E_h(t) := \frac{1}{2}\ll h(t,\cdot,\cdot )\rl^2_V + \e a\la u(t,\cdot), \pt_x\s(t,\cdot) \ra, 
\end{align*}
where $h$ fulfills equation \eqref{micro eq}, $\sigma$ and $u$ defined by equation \eqref{eq: def s}, i.e.,
\begin{align*}
\sigma(t,x)=\int_\R h(t,x,v)\sqrt{M}(v) dv, \; u(t,x)=\int_\R v h(t,x,v)\sqrt{M}(v) dv.
\end{align*}

\begin{lemma}
	\label{lemma: energy est}
Assume that Assumption \ref{cond: E} holds true and assume $1 \geq \epsilon>0$ and let any matrix $K(\e)\in \R^{2\times 2}$ be given. Then,   $E_h $ defined by equation \eqref{eq: def E_h}
fulfills the following estimates  for any  solution $h$ to (\ref{micro eq}) with initial data $h_0(\cdot,\cdot)\in V$ as in equation \eqref{h IC} and boundary conditions~\eqref{eq: general cond},
	\begin{equation}
	\label{eq: energy est}
	\begin{aligned}
	\pt_tE_h + \xi \ll h\rl^2_{V}\leq& - \int_{\R} \frac{v}{2} \(h^2(t,1,v) - h^2(t,0,v)\) dv \\
	&- \int_{\R} \frac{v}{2} \((\pt_xh)^2(t,1,v) - (\pt_xh)^2(t,0,v)\) dv \\
	&-a \(u\pt_xu(t,1) - u\pt_xu(t,0)\)+ \frac{aC_s}{2\e^2}\( \int_0^t\(u(s,1) - u(s, 0)\) ds\)^2.
	\end{aligned}
	\end{equation}
 The constant $\xi$ is given by  
	\begin{align}
	&\xi = \min\l\{ \frac{\lam-C_E-4a}{\e^2},  \frac{a(3C_s-2C_E)-4C_E}{8} \r\}>0,\label{eq: const xi}
	\end{align}
	and $a$ and $C_E$ are chosen such that 
	\begin{align}
	&\frac{4C_E}{3C_s-2C_E} < a < \frac{\lam - C_E}{4},\qd C_E \leq \frac{\lam C_s}{8}, \label{eq: cosnt a C_E}
	\end{align}
	hold. The constants  $C_s,\lam$ are defined in \eqref{sobolev const} and \eqref{coercivity}, respectively.
\end{lemma}
\begin{proof}
	The proof is divided into two parts. 
	
Consider equation 	\eqref{micro eq}, multiply by $\frac{h}2$, integrate on $\Omega$ and use the estimate
(\ref{coercivity}) to obtain 
	\begin{equation}
	\label{eq: micro_1}
	\frac{\e}{2} \pt_t \ll h \rl^2 + \int \frac{v}{2} \pt_x (h^2) dxdv + \frac{\lam}{\e} \ll (1-\Pi)h \rl_\o^2 \leq \la E\(\pt_v - \frac{v}{2}\)h, h \ra. 
	\end{equation}
	In the following computation we use $<g,h>=\int_\Omega g h dxdv$ for simplified notations.  We use the decomposition of $h$ into the macroscopic kernel $\Pi h(t,x,v) = \s(t,x)M(v)$ and its complement to simplify the right-hand side,
\begin{align}
	\la E\(\pt_v - \frac{v}{2}\)h, h \ra  =&  \la E\(\pt_v - \frac{v}{2}\)h, \s\sM \ra + \la E\(\pt_v - \frac{v}{2}\)h, (1-\Pi)h\ra\nonumber\\
	= & \la Eh, \(-\pt_v - \frac{v}{2}\)\s\sM \ra + \la E h, \(-\pt_v - \frac{v}{2}\)(1-\Pi)h\ra\nonumber\\
	\leq & \la Eh, \(\frac{v}{2} - \frac{v}{2}\)\s\sM \ra + \ll E \rl_{L^\infty_x} \(\frac{\e}{2}\ll h \rl^2 + \frac{1}{2\e}\ll (\pt_v + \frac{v}{2})(1-\Pi)h \rl^2 \)\nonumber\\
	\leq &\ll E \rl_{L^\infty_x} \(\frac{\e}{2}\ll h \rl^2 + \frac{1}{2\e}\ll (1-\Pi)h \rl_\o^2 \).\label{eq: nonlinear est}
	\end{align}
	Note that we only integrate by parts in $v$ which is defined in the whole space, therefore no boundary conditions appear. Further, note that we used in the last line Young's inequality with the same fixed value of $\e$ 
	as in the equation. The estimate for  terms involving the electric field $E$ yields in equation \eqref{eq: micro_1} the following  preliminary result 
	\begin{equation}
	\label{eq: micro_2}
	\begin{aligned}
	\frac{\e}{2} \pt_t \ll h \rl^2 + \frac{\lam}{\e} \ll (1-\Pi)h \rl_\o^2 \leq& \frac{1}{2}\ll E \rl_{L^\infty_x} \(\e\ll h \rl^2 + \frac{1}{\e} \ll (1-\Pi)h \rl_\o^2 \) \\
	&- \int \frac{v}{2}\(h^2(t,1,v) - h^2(t,0,v)\) dv . 
	\end{aligned}
	\end{equation}
	The previous computation is repeated for the spatial derivative of (\ref{micro eq}). Integration on $\Omega$ yields 	
	\begin{equation}
	\label{eq: micro_3}
	\begin{aligned}
	&\frac{\e}{2} \pt_t \ll \pt_xh \rl^2 + \int \frac{v}{2} \pt_x ((\pt_xh)^2) dxdv + \frac{\lam}{\e} \ll (1-\Pi)\pt_xh \rl_\o^2 \\
	\leq& \la (\pt_xE)\(\pt_v - \frac{v}{2}\)h, \pt_xh \ra + \la E\(\pt_v - \frac{v}{2}\)\pt_xh, \pt_xh \ra. 
	\end{aligned}
	\end{equation}
	
	The second part on the right--hand side can be bounded in a similar fashion as in equation \eqref{eq: nonlinear est}. For first part of the right--hand side we apply Young's inequality and use the fact that $\partial_x E$ is also bounded. Hence, we obtain 
	\begin{equation}
	\label{eq: nonlinear est_2}
	\begin{aligned}
	& \la \pt_xE\(\pt_v - \frac{v}{2}\)h, \pt_xh \ra + \la E\(\pt_v - \frac{v}{2}\)\pt_xh, \pt_xh \ra \\
	\leq& \frac{1}{2} \ll \pt_xE\rl_{L^\infty_x} \( \e\ll h \rl^2 +  \frac{1}{\e}\ll (1-\Pi)\pt_xh \rl^2_\o\) + \frac{1}{2}\ll E \rl_{L^\infty_x} \(\e \ll \pt_x h\rl^2 + \frac{1}{\e}\ll (1-\Pi)\pt_xh \rl_\o^2\).
	\end{aligned}
	\end{equation}
	Combining the estimate on the terms of electric field with equation \eqref{eq: micro_3} we obtain 
	a bound on the norm of $\pt_x h$ similar to \eqref{eq: micro_2},
	\begin{equation}
	\label{eq: micro_4}
	\begin{aligned}
	&\frac{\e}{2} \pt_t \ll \pt_xh \rl^2  + \frac{\lam}{\e} \ll (1-\Pi)\pt_xh \rl_\o^2 \\
	\leq&  \frac{1}{2} \ll \pt_xE\rl_{L^\infty_x} \( \e\ll h \rl^2 +  \frac{1}{\e}\ll (1-\Pi)\pt_xh \rl^2_\o\) + \frac{1}{2}\ll E \rl_{L^\infty_x} \(\e \ll \pt_x h\rl^2 + \frac{1}{\e}\ll (1-\Pi)\pt_xh \rl_\o^2\)\\
	&- \int \frac{v}{2}\((\pt_xh)^2(t,1,v) - (\pt_xh)^2(t,0,v)\) dv  . 
	\end{aligned}
	\end{equation}
	Both bounds  (\ref{eq: micro_2}) and (\ref{eq: micro_4}) together give a bound in the $V-$ and $V,\omega$-norm defined in Section \ref{notation},
	\begin{equation}
	\label{eq: micro_5}
	\begin{aligned}
	&\frac{\e}{2} \pt_t \ll h \rl_V^2  + \frac{\lam}{\e} \ll (1-\Pi)h \rl_{V,\o}^2 \\
	\leq&  \frac{C_E}{2} \(\e \ll h\rl_V^2 + \frac{1}{\e}\ll (1-\Pi)h \rl_{V,\o}^2\) - \int \frac{v}{2}\(h^2(t,1,v) - h^2(t,0,v)\) dv  \\
	&- \int \frac{v}{2}\((\pt_xh)^2(t,1,v) - (\pt_xh)^2(t,0,v)\) dv  . 
	\end{aligned}
	\end{equation}
	In estimate \eqref{eq: micro_5} we estimate by $C_E$ the maximum of the $L^\infty-$ norm of the electric field and its	spatial derivative as stated in Assumption \ref{cond: E}. 
	Finally, note that $\| h\|^2_V= \ll \s \rl^2_V + \ll (1-\Pi)h \rl_V^2$, see Section \ref{notation}. Hence, 
	we obtain after multiplication by $\frac1\e$ 
	\begin{equation}
	\label{eq: micro_6}
	\begin{aligned}
	\frac{1}{2} \pt_t \ll h \rl_V^2  + \frac{\(\lam - C_E\)}{\e^2} \ll (1-\Pi)h \rl_{V,\o}^2\leq&  \frac{C_E}{2}  \ll \s\rl_V^2  - \frac{1}{\e}\int \frac{v}{2}\(h^2(t,1,v) - h^2(t,0,v)\) dv  \\
	&- \frac{1}{\e}\int \frac{v}{2}\((\pt_xh)^2(t,1,v) - (\pt_xh)^2(t,0,v)\) dv  . 
	\end{aligned}
	\end{equation}
	This completes the first part of the proof. 	
	
	In the second step we derive an estimate on the macroscopic quantities $\s$ and $u.$ Recall that, provided $h$ is a solution to equation \eqref{micro eq}, the corresponding macroscopic quantities fulfill \eqref{continuity} and equation \eqref{macro eq}, i.e., 
	\begin{align}
	&\e \pt_t \s + \pt_x u = 0, \\
	&\e \pt_t u + \pt_x \s + \int v^2 \sqrt{M} (1-\Pi) \partial_x h dv + \frac{1}{\e} u = - E \s. \label{mh ss}
	\end{align}
		Upon multiplication of $\pt_x\s$ to equation  (\ref{mh ss}), and integration on $x$, we obtain  
	\begin{equation}
	\label{eq: macro_1}
	\begin{aligned}
	&\e\la \pt_t u, \pt_x\s \ra + \ll \pt_x\s \rl^2 + \la v^2\sM (1-\Pi)\pt_xh, \pt_x\s \ra + \frac{1}{\e} \la u, \pt_x\s \ra_x = -\la E\s, \pt_x\s\ra_x 
	\end{aligned}
\end{equation}
Note that here $<\cdot,\cdot>_x$ denotes the integration on $x$ only. By Young's inequality we bound the 
 third and fourth terms in equation \eqref{eq: macro_1} as follows: 
\begin{equation}
\label{eq: term_2}
\begin{aligned}
\la v^2\sM (1-\Pi)\pt_xh, \pt_x\s \ra &\leq \int  \ll v (1-\Pi)\pt_xh \rl_{L^2_v}\ll v\sM \rl_{L^2_v}  \pt_x\s dx \\
 &\leq \frac{1}{2}\ll v\sM \rl_{L^2_v} \(4\ll v(1-\Pi)\pt_xh \rl^2 +  \frac{1}{4}\ll \pt_x\s\rl^2\) \\
 &\leq 2\ll (1-\Pi)\pt_xh \rl_\o^2 + \frac{1}{8}\ll \pt_x\s \rl^2;\\
\frac{1}{\e}\la u, \pt_x\s \ra_x & \leq \frac{1}{2\e}\(\frac4\e\ll u \rl^2 +\frac\e4 \ll \pt_x\s \rl^2\) \leq \frac{2}{\e^2}\ll u\rl^2 + \frac{1}{8}\ll \pt_x\s \rl^2;\\
\la E\s, \pt_x\s\ra_x &  \leq \frac{1}{2}\ll E \rl_{L^\infty_x} (\ll \s\rl^2+ \ll \pt_x\s \rl^2)\leq  \frac{C_E}{4} \ll \s \rl_V^2 ,
\end{aligned}
\end{equation}
where we apply assumption \eqref{bdd of E} in the last inequality.
Using the continuity equation \eqref{continuity} we simplify  the term $<\pt_t u, \pt_x \s>$, 	
	\begin{equation}
	\label{eq: term_1}
	\begin{aligned}
	&\e\la\pt_tu, \pt_x\s \ra = \e\pt_t \la u, \pt_x\s \ra + \la u, \pt_x(  \pt_x u ) \ra =  \e\pt_t \la u, \pt_x\s \ra +  \la u, \pt_x^2u \ra  \\
	=&  \e\pt_t \int  u  \pt_x\s dx +    ( u\pt_xu) (t,1) - (u\pt_xu) (t,0)  - \ll \pt_xu \rl^2.
	\end{aligned}
	\end{equation}
Combining the previous estimates we can simplify (\ref{eq: macro_1}) as 
	\begin{equation}
	\label{eq: macro_2}
	\begin{aligned}
	&\e \pt_t \la u, \pt_x\s \ra   + \ll \pt_x\s \rl^2 \leq  - ( u\pt_xu) (t,1) - (u\pt_xu) (t,0)  + \| \pt_x u\|^2  
	+ \\
	& + \frac{C_E}4 \|\s\|_V^2  +
	\frac{2}{\e^2} \|u\|^2 + \frac28 \| \pt_x \s \|^2  + 2 \| (1-\Pi) \pt_x h\|^2_\omega,
	\\
	&\e\pt_t\la u, \pt_x\s \ra + \frac{3}{4}\ll \pt_x\s \rl^2  \leq \frac{C_E}{4}\ll \s \rl^2_{V} +  \frac{4}{\e^2}\ll (1-\Pi)h \rl_{V,,\omega}^2- \(u\pt_xu (t,1,z) - u\pt_xu(t,0,z)\)  ,
	\end{aligned}
	\end{equation}
	where $\| u \|^2 \leq  \ll u \rl_V^2 \leq \ll (1-\Pi) h \rl_{V,\o}^2$ is used. 
	
	Next we turn to estimates of $\s$ and $\pt_x \s.$  By  Poincare's inequality (\ref{sobolev const}), there exists a constant $C_s\leq 1$, such that
	\begin{align}
	\ll \pt_x\s \rl^2 \geq \frac{C_s}2 \ll \s - \int_0^1 \s dx\rl^2 \geq \frac{C_s}2 \(\ll \s \rl_V^2 -\( \int_0^1 \s \, dx\)^2\). \label{eq: poincare}
	\end{align}
In order to estimate also $\s$ we integrate the continuity equation \eqref{continuity} with respect to $x$, 
	\begin{equation}
	\begin{aligned}
	&\e\pt_t\int_0^1\s dx = -\(u(t,1) - u(t, 0)\)\\
	 \implies 
	&\int_0^1\s(t,x) dx - \int_0^1\s_0(x) dx =- \frac{1}{\e} \int_0^t\(u(s,1) - u(s, 0)\) ds.
	\end{aligned}
	\end{equation}
	Due to Assumption \ref{ass initial} we have that $\int_\Omega h_0(x,v) dx dv=0.$ Therefore, 
	 $\int_0^1 \s_0(x) dx =0$ and inequality \eqref{eq: poincare} simplifies to 
	\begin{equation}
	\begin{aligned}
	\| \pt_x \sigma \|^2 \geq \frac{C_s}2 \(\ll \s \rl_V^2 -  \frac{1}{\e^2}\( \int_0^t\(u(s,1) - u(s, 0)\) ds\)^2
\). 
	\end{aligned}
	\end{equation}
	Inserting the above inequality to \eqref{eq: macro_2} gives: 
	\begin{align}
		\e\pt_t\la u, \pt_x\s \ra + \frac{ 3 C_s}{8} 
	\(\ll \s \rl_V^2 -  \frac{1}{\e^2}\( \int_0^t\(u(s,1) - u(s, 0)\) ds\)^2
	\)\nonumber\\ \leq 
	\frac{C_E}{4}\ll \s \rl^2_{V} +  \frac{4}{\e^2}\ll (1-\Pi)h \rl_{V,\omega}^2- \(u\pt_xu (t,1,z) - u\pt_xu(t,0,z)\). \nonumber
	\end{align}
	This in turn implies the final estimate for $\s$.  
	\begin{align}
	\e\pt_t\la u, \pt_x\s \ra + \frac{ 3 C_s - 2 C_E  }{8} \| \s\|_V^2 
	&\leq  \frac{C_s}{2 \e^2}\( \int_0^t\(u(s,1) - u(s, 0)\) ds\)^2 +  \nonumber\\ 
	& + \frac{4}{\e^2} \ll (1-\Pi)h \rl_{V,\omega}^2  - \(u\pt_xu (t,1,z) - u\pt_xu(t,0,z)\). 
		\label{eq: macro_3}\end{align}
In the last step of the proof we add the estimate for $h$ obtained in equation \eqref{eq: micro_6}
and the previous estimate \eqref{eq: macro_3} (multiplied by $a \geq 0$). 

\begin{align}
&\frac{1}{2} \pt_t \ll h \rl_V^2  +a \e  \pt_t\la u, \pt_x\s \ra + a \frac{ 3 C_s - 2 C_E  }{8} \| \s\|_V^2  + 
\frac{\(\lam - C_E\)}{\e^2} \ll (1-\Pi)h \rl_{V,\o}^2   \nonumber\\ 
\leq&
  \frac{C_E}{2}  \ll \s\rl_V^2 + \frac{4a }{\e^2} \ll (1-\Pi)h \rl_{V,\omega}^2\nonumber \\ 
  &+ \frac{C_s a}{2 \e^2}\( \int_0^t\(u(s,1) - u(s, 0)\) ds\)^2 + 
 - a \(u\pt_xu (t,1,z) - u\pt_xu(t,0,z)\)\nonumber \\
&- \frac{1}{\e}\int \frac{v}{2}\((\pt_xh)^2(t,1,v) - (\pt_xh)^2(t,0,v)\) dv.
\end{align}
 
 This implies that for $\xi$ sufficiently small 
\begin{align} 
&\partial_t E_h(t)  + \xi \| h \|_{V}^2 \leq \partial_t E_h(t)  + \xi \| h \|_{V,\omega}^2\nonumber \\
\leq &\partial_t E_h(t)  + \(  a \frac{ 3 C_s - 2 C_E   }{8} - \frac{C_E}2 \) \| \s\|_V^2  + 
\frac{\(\lam - C_E - 4 a \)}{\e^2} \ll (1-\Pi)h \rl_{V,\o}^2 \nonumber \\
\leq &   - \frac{1}{\e}\int \frac{v}{2}\(h^2(t,1,v) - h^2(t,0,v)\) dv  
- \frac{1}{\e}\int \frac{v}{2}\((\pt_xh)^2(t,1,v) - (\pt_xh)^2(t,0,v)\) dv  \nonumber\\
&+\frac{C_s a}{2 \e^2}\( \int_0^t\(u(s,1) - u(s, 0)\) ds\)^2  - a \(u\pt_xu (t,1,z) - u\pt_xu(t,0,z)\). \label{eqE_h}
\end{align}
More precisely, $\xi$ can be chosen as 
	\begin{equation*}
	\xi = \min\l\{ \frac{\lam-C_E-4a}{\e^2},  \frac{a(3C_s-2C_E)-4C_E}{8} \r\}. 
	\end{equation*}
	In order to obtain  exponential decay of $\ll h\rl_V^2$, we require $\xi>0$, which implies the following condition on $a:$ 
	\begin{align}
	\label{eq: cosnt a}
	&\frac{4C_E}{3C_s-2C_E} < a < \frac{\lam - C_E}{4}.
	\end{align}
	In order for $a$ to exist we require that the previous bounds are ordered. A sufficient condition for 
	$\frac{4C_E}{3C_s-2C_E} < \frac{\lam - C_E}{4}$  is for example given by 
	\begin{align}
	\label{eq: cosnt C_E}
	&C_E \leq \frac{\lam C_s}{8}
	\end{align}
	since $\lam \leq 1, C_s \leq \frac{1}{2}$. 
\end{proof}

\subsection{Constraints on the feedback control.}
\label{sec: boundary cond}
In the previous section we derived an estimate for the decay of $E_h.$ The final equation still includes
boundary terms. In the following two Lemmas, we derive conditions on $K(\e)$ such that in the estimate \eqref{eq: energy est}   terms involving the boundary values of $h$ at $x\in \{0,1\}$   are non-positive. 
This fact  allows to conclude the proofs of Theorem \ref{thm: general} and Theorem \ref{thm: general limit}, respectively. Note that although the following lemmas are stated for function $f$, they also remain valid
 for $h$. In fact, at the boundary $f$ and $h$ are the same up to a scaling by $\sqrt{M(v)}.$

\begin{lemma}
	\label{lemma: bdy for u}
Let Assumption \ref{ass initial} hold true and let $\e\geq 0.$ Assume that $f$ solves the VFP \eqref{VFP} with initial condition $f(0,x,v)=f_0(x,v)$ for some $f_0(\cdot,\cdot) \in V$ and boundary conditions \eqref{eq: general cond_f}. 

If  the entries of $K(\e)$ defined by \eqref{def of K} fulfill  
	\begin{equation}
	\label{constraint 1}
	\koo(\e)+\kol(\e) = 1,\qd \klo(\e) + \kll(\e) = 1,
	\end{equation}
	then the boundary flux $u(t,x)=\int_\R v f(t,x,v) dv$  for $x \in \{0,1 \}$ fulfills
	\begin{equation}
	\label{bdy of u}
	u(t,1)=u(t,0).
	\end{equation}
\end{lemma}
\begin{proof}
A direct computation shows that 
	\begin{equation}
	\label{u bdy_1}
	\begin{aligned}
	&u(t,0) = \int_0^\infty v f(t,0,v) dv + \int_{-\infty}^0 vf(t,0,v)dv\\  
	=&  \int_0^\infty v \koo f(t,0,-v) dv +\int_0^\infty v \klo f(t,1,v) dv + \int_{-\infty}^0 vf(t,0,v)dv\\
	=& \(1-\koo\)\int^0_{-\infty} v  f(t,0,v) dv+ \klo\int_0^\infty v f(t,1,v) dv.
	\end{aligned}
	\end{equation}
	Similarly, we obtain  
	\begin{equation}
	\label{u bdy_2}
	\begin{aligned}
	&u(t,1) = \(1-\kll\)\int_0^{\infty} v  f(t,1,v) dv + \kol\int_{-\infty}^0 v  f(t,0, v) dv.
	\end{aligned}
	\end{equation}
	Subtracting (\ref{u bdy_2}) from (\ref{u bdy_1}) and inserting the assumption (\ref{u bdy_1}) gives
	\begin{equation}
	\label{u bdy_1}
	\begin{aligned}
	&u(t,0) - u(t,1)
	= \(1-\koo-\kol\)\int^0_{-\infty} v  f(t,0,v) dv+ \(\klo-1+\kll\)\int_0^\infty v f(t,1,v) dv = 0.
	\end{aligned}
	\end{equation}
	This finishes the proof. 
\end{proof}

Finally, we derive an estimate on the boundary contributions in the estimate \eqref{eq: energy est}. 
\begin{lemma}\label{lemma: negative RHS est}
	
Let Assumption \ref{ass initial} hold true and let $\e\geq 0.$ Assume that $h$ solves the VFP \eqref{VFP2} with initial condition $h(0,x,v)=\frac{f_0(x,v)}{\sqrt{M}}$ for some $f_0(\cdot,\cdot) \in V$ and boundary conditions \eqref{eq: general cond}. 
 
 Let assumption \eqref{const 2} hold true, i.e., 
  \begin{equation}
	\label{eq: relation_1}
	\koo(\e) + \kol(\e) = 1, \qd \kll(\e)+\klo(\e) = 1. 
	\end{equation}
	
	Then, $h$ fulfills the estimate \eqref{eq: energy est} and we have  
	\begin{equation*}
	-\int_\R \frac{v}{2}\l[h^2\r]_0^1 dv-\int_\R \frac{v}{2}\l[(\pt_xh)^2\r]_0^1  dv - a\l[u\pt_xu\r]_0^1 \leq I(t,\e),
	\end{equation*}
	where  by $[g]_0^1 = g(t,1,v)-g(t,0,v)$. The term $I(t,\e)$ is given by  
	\begin{equation}
	\label{def of I(t,e)}
	\begin{aligned}
	I(t,\e) 
	= &-  2\koo(1- \koo)\(\A+\Ax\)-2\kll(1- \kll)\(\B+\Bx\)\\
	&+2\(\lv\kll(1-\koo)\rv + \lv\koo(1-\kll)\rv\)\(\sqrt{\A\B}+\sqrt{\Ax\Bx}\)\\
	&+4a\(\lv 1-\kll \rv\sqrt{\B} + \lv 1-\koo\rv\sqrt{\A}\)\(\lv\koo\rv\sqrt{\Ax}+\lv\kll\rv
	\sqrt{\Bx}\)
	\end{aligned}
	\end{equation}
	where $\A(t),\B(t),\Ax(t),\Bx(t)$ are positive functions of time $t$ defined by equation (\ref{def of ABCD}).
\end{lemma}

\begin{proof}
	In the proof we drop the dependence of $K(\e)$ on $\e$ for notation clarity. 	By the boundary condition (\ref{eq: general cond})  we obtain  
	\begin{align}
	&-\int_\R  \frac{v}{2} h^2(t,1,v) dv = -\int_0^\infty  \frac{v}{2} h^2(t,1,v) dv - \int_{-\infty}^0  \frac{v}{2} \(\kll h(t,1,-v) + \kol h(t,0,v)\)^2  dv\nonumber\\
	=& -(1- \kll^2)\B +  \kol^2\A  -  2\kll\kol\int_{-\infty}^0  \frac{v}{2} h(t,1,-v) h(t,0,v)  dv,\nonumber
	\end{align}
	where we expanded the quadratic term and changed the variable  $-v$ to $v$ within the  term $h(t,1,-v)$. Similarly,  
	\begin{align}
	&\int_\R  \frac{v}{2} h^2(t,0,v) dv = \int_{-\infty}^0 \frac{v}{2} h^2(t,0,v) dv +\int_0^\infty \frac{v}{2}\(\koo h(t,0,-v) +\klo h(t,1,v)\)^2dv\nonumber\\
	=& -(1- \koo^2)\A + \klo^2 \B -  2\koo\klo\int^0_{-\infty}  \frac{v}{2} h(t,1,-v) h(t,0,v)  dv.  \nonumber
	\end{align}
	Writing $v=-\sqrt{v}\sqrt{v}$ we obtain 
	\begin{equation*}
	\begin{aligned}
	&\lv\int_{-\infty}^0  \frac{v}{2} h(t,1,-v) h(t,0,v)  dv\rv\leq \sqrt{\int_{-\infty}^0  \frac{-v}{2} h^2(t,1,-v) dv}  \sqrt{\int_{-\infty}^0  \frac{-v}{2}h^2(t,0,v)  dv } = \sqrt{AB}.
	\end{aligned}
	\end{equation*}
	Combining the previous estimates yields
	\begin{align}
	&-\int_\R  \frac{v}{2} \(h^2(t,1,v) -h^2(t,0,v)\) dv\nonumber\\
	=&-(1- \kll^2 - \klo^2)\B -  (1- \koo^2 - \kol^2)\A\nonumber\\
	& -  2\kll\kol\int_{-\infty}^0  \frac{v}{2} h(t,1,-v) h(t,0,v)  dv +  2\koo\klo\int^{\infty}_0  \frac{v}{2} h(t,1,v) h(t,0,-v)  dv,  \nonumber\\
	\leq &- (1- \koo^2 - \kol^2 )\A-(1- \kll^2 - \klo^2 )\B
	+ 2\(\lv \kll\kol + \koo\klo \rv\) \sqrt{\A\B}\nonumber\\
	=&-  2\koo(1- \koo)\A-2\kll(1- \kll)\B+2\(\lv\kll(1-\koo) +\koo(1-\kll)\rv\)\sqrt{\A\B}, \label{pf: add_1}
	\end{align}
where we have used  relation (\ref{eq: relation_1}). In order to estimate the  term $-\int \frac{v}{2}((\pt_xh)^2(t,1,v,z) - (\pt_xh)^2(t,0,v,z)) dv$ we recall the  boundary condition of $\pt_xh$ in Lemma \ref{lemma: bdy pt_x}:
	\begin{align}
	&-\int_\R  \frac{v}{2} (\pt_xh)^2(t,1,v) dv \nonumber\\
	=& -\int_0^\infty  \frac{v}{2} (\pt_xh)^2(t,1,v) dv - \int_{-\infty}^0  \frac{v}{2} \(-\kll \pt_xh(t,1,-v) + \kol \pt_xh(t,0,v)\)^2  dv\nonumber\\
	=& -(1- \kll^2)\Bx +  \kol^2\Ax  +  2\kll\kol\int_{-\infty}^0  \frac{v}{2} \pt_xh(t,1,-v) \pt_xh(t,0,v)  dv,\nonumber
	\end{align}
	and similarly
	\begin{align}
	&\int_\R  \frac{v}{2} (\pt_xh)^2(t,0,v) dv = -(1- \koo^2)\Ax +  \klo^2 \Bx +  2\koo\klo\int^0_{-\infty}  \frac{v}{2} \pt_xh(t,1,-v) \pt_xh(t,0,v)  dv.  \nonumber
	\end{align}
	The above  inequalities  imply that 
	\begin{equation}
	\label{pf: add_2}
	\begin{aligned}
	&-\int_\R  \frac{v}{2} \((\pt_xh)^2(t,1,v) -(\pt_xh)^2(t,0,v)\) dv\\
	\leq &-  (1- \koo^2 - \kol^2 )\Ax-(1- \kll^2 - \klo^2 )\Bx+2\(\lv\kll\kol + \koo\klo\rv\)\sqrt{\Ax\Bx} \\
	=&-  2\koo(1- \koo)\Ax-2\kll(1- \kll)\Bx+2\(\lv\kll(1-\koo) +\koo(1-\kll)\rv\)\sqrt{\Ax\Bx}.
	\end{aligned}
	\end{equation}
	Finally, we  bound $-u\pt_xu(t,1) + u\pt_x(t,0)$. Use the definition of $u$ in terms of $h$ as 
	\begin{equation*}
	\begin{aligned}
	&-u\pt_xu(t,1)+u\pt_xu(t,0)  \\
	=&-\((1-\kll)\int_0^{\infty} v\sM h(t,1,v)dv + \kol \int_{-\infty}^0 v\sM h(t,0,v)dv\) \\
	&\( (1+\kll)\int_0^{\infty} v\sM \pt_xh(t,1,v)dv + \kol\int_{-\infty}^0v\sM \pt_xh(t,0,v)dv\)\\
	&+\((1-\koo)\int_{-\infty}^0 v\sM h(t,0,v)dv + \klo \int_0^{\infty} v\sM h(t,1,v)dv\) \\
	&\( (1+\koo)\int_{-\infty}^0 v\sM \pt_xh(t,0,v)dv + \klo\int_0^{\infty}v\sM \pt_xh(t,1,v)dv\).
	\end{aligned}
	\end{equation*}
Due to the assumption  $(1-\kll) = \klo  = (1-\koo)$,  therefore, 
	\begin{equation}
	\label{pf: upx_u}
	\begin{aligned}
	&-u\pt_xu(t,1)+u\pt_xu(t,0)  \\
	=&\((1-\kll)\int_0^{\infty} v\sM h(t,1,v)dv + (1-\koo) \int_{-\infty}^0 v\sM h(t,0,v)dv\) \\
	&\(-2\kll\int_0^{\infty} v\sM \pt_xh(t,1,v)dv+2\koo\int_{-\infty}^0v\sM \pt_xh(t,0,v)dv\).
	\end{aligned}
	\end{equation}
	The integral on $v \sqrt{M}h$ is bounded  
	\begin{equation*}
	\begin{aligned}
	&\int_0^{\infty} v\sM h(t,1,v)dv \leq \sqrt{\int_0^{\infty} vM \, dv} \sqrt{\int_0^\infty vh^2(t,1,v)dv} =\sqrt{2\B}.
	\end{aligned}
	\end{equation*} 
	The $x=0$ case is handled similarly. Therefore  (\ref{pf: upx_u}) has an upper bound given by 
	\begin{equation}
	\label{pf: add_3}
	\begin{aligned}
	&-u\pt_xu(t,1)+u\pt_xu(t,0)
	\leq  4\(\lv 1-\kll \rv\sqrt{\B} + \lv 1-\koo\rv\sqrt{\A}\)\(\lv\koo\rv\sqrt{\Ax}+\lv\kll\rv
	\sqrt{\Bx}\).
	\end{aligned}
	\end{equation}
	Combining equations  (\ref{pf: add_1}), (\ref{pf: add_2}) and (\ref{pf: add_3})  we arrive at
	\begin{equation*}
	\begin{aligned}
	&-\int_\R \frac{v}{2}\l[h^2\r]_0^1 dv-\int_\R \frac{v}{2}\l[(\pt_xh)^2\r]_0^1  dv - a\l[u\pt_xu\r]_0^1\\
	\leq &-  2\koo(1- \koo)\(\A+\Ax\)-2\kll(1- \kll)\(\B+\Bx\)\\
	&+2\(\lv\kll(1-\koo)\rv + \lv\koo(1-\kll)\rv\)\(\sqrt{\A\B}+\sqrt{\Ax\Bx}\)\\
	&+4a\(\lv 1-\kll \rv\sqrt{\B} + \lv 1-\koo\rv\sqrt{\A}\)\(\lv\koo\rv\sqrt{\Ax}+\lv\kll\rv
	\sqrt{\Bx}\).
	\end{aligned}
	\end{equation*}
\end{proof}

\section{Conclusion.}
In this paper, we study feedback boundary conditions to guarantee   stabilization of solutions to the Vlasov--Fokker--Planck equation. Using a novel Lyapunov function we are able to show 
that different linear feedback boundary conditions damp out perturbations  of steady states exponentially fast in time.  We discuss boundary conditions for the case where the electric field is suitably small as well as boundary conditions for the potentially large electric fields. Also, we study the interplay of the hydrodynamic limit with the derived feedback laws. In the hydrodynamic limit only periodic or reflective boundary conditions guarantee stabilization of steady states. 

\section*{Acknowledgments.}  
This work has been supported by DFG/HE5386/15,18,19 as well as the DFG funded graduate school GRK2326/20021702.


\medskip
Received xxxx 20xx; revised xxxx 20xx.
\medskip

\end{document}